\documentclass[a4paper, 12pt]{article}
\usepackage{amsmath}
\usepackage{amsfonts}
\usepackage{amssymb}
\usepackage[english]{babel}
\usepackage{graphicx}
\usepackage{amsthm}
\usepackage[title]{appendix}

\newtheorem{thm}{Theorem}[section]

\newtheorem{lem}[thm]{Lemma}
\newtheorem{prop}[thm]{Proposition}

\theoremstyle{definition}
\newtheorem{defn}[thm]{Definition}

\theoremstyle{remark}
\newtheorem{rem}{Remark}[section]

\numberwithin{equation}{section}

\newcommand{\F}{\mathbb{F}}

\newcommand{\R}{\mathbb{R}}

\usepackage{tikz}
\usetikzlibrary{shapes.geometric}
\usetikzlibrary{decorations.markings,arrows}

\begin{document}

\title{Designs over finite fields by difference methods}

\author{Marco Buratti \thanks{Dipartimento di Matematica e Informatica, Universit\`a di Perugia, via Vanvitelli 1 - 06123 Italy, email: buratti@dmi.unipg.it}
\quad\quad
Anamari Naki\'c \thanks{University of Zagreb, Faculty of Electrical Engineering and Computing, Unska 3, 10000 Zagreb, Croatia, email: anamari.nakic@fer.hr}}

\maketitle
\begin{abstract}
\noindent
One of the very first results about designs over finite fields, by S.
Thomas, is the existence of a cyclic 2-$(n,3,7)$ design over
$\mathbb{F}_{2}$ for every integer $n$ coprime with 6. Here, by means
of difference methods, we reprove and improve a little bit this result
showing that it is true, more generally, for every odd $n$. In this way,
we also find the first infinite family of non-trivial cyclic group
divisible designs over $\mathbb{F}_{2}$.
\end{abstract}


\eject
\section{Introduction}

In this paper we adapt very well known difference methods to the
construction of \emph{designs over finite fields}. Our main result will
be the existence of a cyclic 2-$(n,3,7)$ design over $\mathbb{F}_{2}$
for every odd positive $n$. It should be noted that in the case
$n\equiv \pm 1$ (mod 6) our designs are the same found by S. Thomas
\cite{T} a long time ago by means of a geometric approach. Anyway our
proof is algebraic and completely different; we hope it may open the
door to new ideas on this topic. In the new case $n\equiv 3$ (mod 6) we
get designs which, maybe, are not very nice since they are far from
being simple; indeed they have ${2^{n}-1\over 7}$ blocks repeated 7
times. On the other hand, though ``ugly'', these designs allow us to get
the first infinite class of cyclic and simple \emph{group divisible
	designs over finite fields}.

Here we give all prerequisites that are necessary for understanding our
proof of the main result.

\bigskip
$\underline{\mbox{\emph{Classic $2$-designs and group divisible designs}}}$

\medskip
A 2-$(n,k,\lambda )$ design is a pair $(\mathcal{P},\mathcal{B})$ with
$\mathcal{P}$ a set of $n$ \emph{points} and $\mathcal{B}$ a multiset
of $k$-subsets (\emph{blocks}) of $\mathcal{P}$ with the property that
any 2-subset of $\mathcal{P}$ is contained in precisely $\lambda $
blocks.

A $(mg,g,k,\lambda )$ \emph{group divisible design}, briefly a
$(mg,g,k,\lambda )$-GDD, is a triple $(\mathcal{P},\mathcal{G},
\mathcal{B})$ with $\mathcal{P}$ a set of $mg$ \emph{points},
$\mathcal{G}$ a partition of $\mathcal{P}$ into $m$ subsets
(\emph{groops})\footnote{Here, following \cite{BJL}, we misspell the
	word ``group'' on purpose in order to avoid confusion with the groups
	understood as algebraic structures.} of size $g$, and $\mathcal{B}$ a
multiset of $k$-subsets (\emph{blocks}) of $\mathcal{P}$ with the two
properties that a block and a groop have at most one common point, and
any two points belonging to distinct groops are contained, together, in
exactly $\lambda $ blocks.

It is clear that a $(n,1,k,\lambda )$-GDD is completely equivalent to
a 2-$(n,k,\lambda )$ design.

An automorphism of a 2-design or group divisible design is a permutation
of its point-set leaving invariant its block-multiset.

A 2-design or group divisible design is said to be \emph{simple} if it
does not have repeated blocks.

\bigskip
$\underline{\mbox{\emph{Cyclic $2$-designs and difference families}}}$

\medskip
A 2-design is said to be \emph{cyclic} if it admits an automorphism
cyclically permuting all its points or, equivalently, if it has a cyclic
automorphism group acting sharply transitively on the points. It is
known that every cyclic 2-design can be described in terms of
differences \cite{AB}. We recall here the difference methods using the
notion of an \emph{ordinary difference family}.

If $B$ is a subset of an additive (resp. multiplicative) group $H$, the
list of differences of $B$ is the multiset $\Delta B$ of all possible
differences $x-y$ (resp. quotients $xy^{-1}$) with $(x,y)$ an ordered
pair of distinct elements of $B$. The \emph{development} of $B$ under
$H$ is the collection $\mathrm{dev} B=\{B*h\;|\;h\in H\}$ where $*$ is
the (additive or multiplicative) operation of $H$.

Note that if stab$(B)$ is the stabilizer of $B$ under the regular right
action of $H$ on itself, then dev$B$ coincides with the orbit of $B$
replicated $|H:\mathrm{stab}(B)|$ times. So dev$B$ coincides with the
orbit of $B$ when stab$(B)$ is trivial.

If $\mathcal{F}$ is a collection of subsets of $H$, then the list of
differences and the development of $\mathcal{F}$ are, respectively, the
multiset sums
\begin{equation*}
	\Delta {\mathcal{F}}:=\biguplus _{B\in {\mathcal{F}}}\Delta B
	\quad
	\quad \mathrm{and}
	\quad \quad
	\mathrm{dev}{\mathcal{F}}:=\biguplus _{B\in {\mathcal{F}}}{\mathrm{dev}}
	B.
\end{equation*}

%d1.1 #&#
\begin{defn}
	Let $H$ be a group of order $n$. A collection $\mathcal{F}$ of
	$k$-subsets of $H$ is an ordinary $(n,k,\lambda )$ difference family if
	the list of differences of $\mathcal{F}$ covers every non-identity
	element of $H$ exactly $\lambda $ times.
\end{defn}

In the following, the adjective ``ordinary'' will be omitted. The members
of a difference family are usually called \emph{base blocks}. Sometimes,
as in \cite{BJL}, it is also required that the base blocks have
trivial stabilizers. We prefer to remove this constraint since it is not
necessary for the validity of the following well known result.

%t1.2 #&#
\begin{thm}
	\label{DF->design}
	If $\mathcal{F}$ is a $(n,k,\lambda )$ difference family in a group
	$H$, then the pair $(H,dev\mathcal{F})$ is a $2$-$(n,k,\lambda )$ design
	admitting an automorphism group isomorphic to $H$ acting sharply
	transitively on the points.
\end{thm}

So, in particular, the existence of a $(n,k,\lambda )$ difference family
in a cyclic group is a sufficient condition for the existence of a
cyclic 2-$(n,k,\lambda )$ design.

%r1.1 #&#
\begin{rem}
	\label{simple}
	The design generated by a difference family $\mathcal{F}$ is simple if
	and only if all the base blocks of $\mathcal{F}$ have trivial stabilizer
	and they belong to pairwise distinct orbits.
\end{rem}

%\bigskip
$\underline{\mbox{\emph{Designs and difference families over $\mathbb{F}_{2}$}}}$

\medskip
As it is standard, we denote by $\mathbb{F}_{q}^{n}$ the $n$-dimensional
vector space over the field $\mathbb{F}_{q}$ of order $q$. The
multiplicative group of a field $\mathbb{F}$ will be denoted by
$\mathbb{F}^{*}$ and the set of non-zero vectors of $\mathbb{F}_{q}
^{n}$ will be often identified with $\mathbb{F}_{q^{n}}^{*}$.

The $q$-analog of a $t$-$(n,k,\lambda )$ design -- also said a $t$-$(n,k,
\lambda )$ design over $\mathbb{F}_{q}$ or $t$-$(n,k,\lambda )_{q}$
design -- is a collection $\mathcal{S}$ of $k$-dimensional subspaces of
$\mathbb{F}_{q}^{n}$ with the property that any $t$-dimensional subspace
of $\mathbb{F}_{q}^{n}$ is contained in exactly $\lambda $ members of
$\mathcal{S}$. For the survey on recent results, we refer the reader to \cite{BKW}. The most spectacular design over a finite field, obtained
by Braun et al. \cite{BEOVW}, has parameters 2-$(13,3,1)_{2}$. Its
discovering allowed to disprove the longstanding conjecture according
to which the only Steiner $t$-designs over finite fields are the trivial
ones (spreads).

Here we are interested only in 2-$(n,k,\lambda )$ designs over
$\mathbb{F}_{2}$.

%r1.2 #&#
\begin{rem}
	\label{equivalence}
	Every 2-$(n,k,\lambda )$ design over $\mathbb{F}_{2}$ is completely
	equivalent to %
	%\break
	a\break 2-$(2^{n}-1,2^{k}-1,\lambda )$ design $(\mathbb{F}_{2^{n}}^{*},
	\mathcal{B})$ in the classic sense with the crucial property that
	$B \ \cup \ \{0\}$ is a subspace of the vector space $\mathbb{F}_{2}
	^{n}$ for every $B\in {\mathcal{B}}$.
\end{rem}

Indeed, deleting the zero-vector from each block of a 2-$(n,k,\lambda
)_{2}$ design one gets the block-multiset of a classic 2-$(2^{n}-1,2^{k}-1,
\lambda )$ design with point-set $\mathbb{F}_{2^{n}}^{*}$.

For instance, the mentioned 2-$(13,3,1)_{2}$ design is a classic
2-$(8191,7,1)$ design where the points are the non-zero vectors of
$\mathbb{F}_{2}^{13}$ and where every block is the set of non-zero
vectors of a 3-dimensional subspace of $\mathbb{F}_{2}^{13}$. It is
cyclic since it admits $\mathbb{F}_{2^{13}}^{*}$ as an automorphism
group acting sharply transitively on the points. The authors found it
by using the famous Kramer-Mesner method and then they proved that it
could be also obtained from a $(8191,7,1)$ difference
family.\footnote{As a matter of fact, there was no need to prove this
	since it is possible to see that every cyclic $2$-$(n,k,\lambda )$
	design with $\gcd (n,k)=1$ is necessarily generated by a $(n,k,\lambda
	)$ difference family.} Of course this difference family has the special
property that all its members are subspaces of $\mathbb{F}_{2}^{13}$
with the zero-vector removed. This naturally leads to the following
definition.

%d1.3 #&#
\begin{defn}
	\label{DF_2}
	A $(n,k,\lambda )$ difference family over $\mathbb{F}_{2}$ or, briefly,
	a $(n,k,\lambda )_{2}$ difference family, is a $(2^{n}-1,2^{k}-1,
	\lambda )$ difference family in $\mathbb{F}_{2^{n}}^{*}$ with the
	property that $B \ \cup \ \{0\}$ is a subspace of $\mathbb{F}_{2}^{n}$
	for every $B\in {\mathcal{F}}$.
\end{defn}

The above terminology is justified by the following.

%p1.4 #&#
\begin{prop}
	\label{DF_2->design_2}
	A $(n,k,\lambda )_{2}$ difference family generates a cyclic $2$-$(n,k,
	\lambda )_{2}$ design.
\end{prop}
\begin{proof}
	Let $\mathcal{F}$ be a $(n,k,\lambda )_{2}$ difference family. By
	Definition~\ref{DF_2}, $\mathcal{F}$ is a $(2^{n}-1,2^{k}-1,\lambda )$
	difference family in $\mathbb{F}_{2^{n}}^{*}$ and then, by Theorem~\ref{DF->design}, the pair $\mathcal{D}=(\mathbb{F}_{2^{n}}^{*},dev
	\mathcal{F})$ is a cyclic $2$-$(2^{n}-1,2^{k}-1,\lambda )$ design. By
	definition of dev$\mathcal{F}$, each block of $\mathcal{D}$ is of the
	form $xB$ with $x\in \mathbb{F}_{2^{n}}^{*}$ and $B\in {\mathcal{F}}$.
	Also, by Definition~\ref{DF_2}, we have that $B \ \cup \ \{0\}$ is a
	subspace of the vector space $\mathbb{F}_{2}^{n}$ so that $xB \ \cup \ \{0\}$ is a subspace of $\mathbb{F}_{2}^{n}$ as well. Thus every
	block of $\mathcal{D}$ is a subspace of $\mathbb{F}_{2}^{n}$ deprived
	of the zero vector. This means, by Remark~\ref{equivalence}, that
	$\mathcal{D}$ can be seen as a $2$-$(n,k,\lambda )_{2}$ design.
\end{proof}

We will use the above proposition to reprove and improve an old result
by S. Thomas \cite{T} about cyclic 2-$(n,3,7)$ designs over
$\mathbb{F}_{2}$.

\bigskip
\underline{\emph{Cyclic group divisible designs and relative difference
		families}}

\medskip
Cyclic group divisible designs -- namely group divisible designs with
an automorphism group acting sharply transitively on the point-set --
can be also described in terms of differences. In particular, some of
them are generated by \emph{relative difference families}.

%d1.5 #&#
\begin{defn}
	\label{RDF}
	Let $G$ be a subgroup of order $g$ of a group $H$ of order $mg$. A
	collection $\mathcal{F}$ of $k$-subsets of $H$ is a $(mg,g,k,\lambda
	)$ difference family if the list of differences of $\mathcal{F}$ does
	not contain any element of $G$ and covers every element of $H\setminus
	G$ exactly $\lambda $ times.
\end{defn}

One usually says that a difference family $\mathcal{F}$ as above is
\emph{relative to $G$}. It is clear that an ordinary difference family
in $H$ can be seen as a difference family relative to the trivial
subgroup of $H$. More specifically, a $(v,k,\lambda )$ difference family
in $H$ is nothing but a $(v,1,k,\lambda )$ difference family.

Here is the ``group-divisible-analog'' of Theorem~\ref{DF->design}
\cite{B}.
%
%t1.6 #&#
\begin{thm}
	\label{RDF->GDD}
	Let $\mathcal{F}$ be a $(mg,g,k,\lambda )$ difference family in $H$
	relative to $G$. Then, if $\mathcal{G}$ is the set of right cosets of
	$G$ in $H$, the triple $(H,\mathcal{G},dev\mathcal{F})$ is a
	$(mg,g,k,\lambda )$-GDD with an automorphism group isomorphic to $H$
	acting sharply transitively on the points.
\end{thm}

So, in particular, the existence of a $(mg,g,k,\lambda )$ difference
family in a cyclic group is a sufficient condition for the existence of
a cyclic $(mg,g,k,\lambda )$ group divisible design.

The GDD generated by a relative difference family $\mathcal{F}$ is
simple if and only if all the base blocks of $\mathcal{F}$ have trivial
stabilizer and they belong to pairwise distinct orbits.

We will need the following very elementary fact.

%p1.7 #&#
\begin{prop}
	\label{mk,k,k}
	Let $\mathcal{F}$ be a $(mk,k,k)$ difference family in $H$ with a base
	block $G$ that is a subgroup of $H$. Then $\mathcal{F}\setminus \{G\}$
	is a $(mk,k,k,k)$ difference family in $H$ relative to $G$.
\end{prop}
\begin{proof}
	It is enough to note that $\Delta G$ is $k$ times the set of
	non-identity elements of $G$.
\end{proof}

%\bigskip
$\underline{\mbox{\emph{Group divisible designs and relative difference
			families over $\mathbb{F}_{2}$}}}$

\medskip
The $q$-analog of a group divisible design is a concept very recently
introduced in \cite{BKKNW}. First recall that a \emph{$g$-spread} of
the vector space $\mathbb{F}_{q}^{n}$ is a set of $g$-dimensional
subspaces covering $\mathbb{F}_{q}^{n}$ and intersecting each other
trivially.

%d1.8 #&#
\begin{defn}
	Let $\mathcal{S}$ be a $g$-spread of $\mathbb{F}_{q}^{n}$ and let
	$\mathcal{T}$ be a collection of $k$-dimensional subspaces of
	$\mathbb{F}_{q}^{n}$. The triple $(\mathbb{F}_{q}^{n},\mathcal{S},
	\mathcal{T})$ is a $(n,g,k,\lambda )$ group divisible design over
	$\mathbb{F}_{q}$, briefly a $(n,g,k,\lambda )_{q}$-GDD, if any
	$2$-dimensional subspace of $\mathbb{F}_{q}^{n}$ is either contained in
	exactly one member of $\mathcal{S}$ or contained in exactly
	$\lambda $ members of $\mathcal{T}$ but not both.
\end{defn}

Note that when $g=1$, then $\mathcal{S}$ is necessarily the set of all
1-dimensional subspaces of $\mathbb{F}_{q}^{n}$ and $\mathcal{T}$ is a
2-$(n,k,\lambda )_{2}$ design.

%r1.3 #&#
\begin{rem}
	\label{equivalenceGDD}
	Every $(mg,g,k,\lambda )$ design over $\mathbb{F}_{2}$ is completely
	equivalent to a $(2^{mg}-1,2^{g}-1,2^{k}-1,\lambda )$-GDD with point-set
	$\mathbb{F}_{2^{mg}}^{*}$ and the properties that the groops are the
	elements -- with the zero-vector removed -- of a $g$-spread, and that
	each block is the set of non-zero vectors of a $k$-dimensional subspace.
\end{rem}

Indeed, deleting the zero-vector from each groop and from each block of
a\break $(mg,g,k,\lambda )_{2}$-GDD one get a classic $(2^{mg}-1,2^{g}-1,2^{k}-1,
\lambda )$-GDD.

%d1.9 #&#
\begin{defn}
	\label{RDF_2}
	A $(mg,g,k,\lambda )_{2}$ difference family over $\mathbb{F}_{2}$,
	briefly a $(mg,g,k,\lambda )_{2}$ difference family, is a $(2^{mg}-1,2^{g}-1,2^{k}-1,
	\lambda )$ difference family in $\mathbb{F}_{2^{mg}}^{*}$ with the
	property that $B \ \cup \ \{0\}$ is a subspace of $\mathbb{F}_{2}^{mg}$
	for every $B\in {\mathcal{F}}$.
\end{defn}

The above terminology is justified by the following result.

%p1.10 #&#
\begin{prop}
	\label{RDF_2->GDD_2}
	Every $(mg,g,k,\lambda )_{2}$ difference family generates a cyclic
	$(mg,g,k,\lambda )_{2}$-GDD.
\end{prop}
\begin{proof}
	Let $\mathcal{F}$ be a $(mg,g,k,\lambda )_{2}$ difference family. So,
	by definition, $\mathcal{F}$ is a $(2^{mg}-1,2^{g}-1,2^{k}-1,\lambda
	)$ difference family in $\mathbb{F}_{2^{mg}}^{*}$. Let $G$ be the
	subgroup of $\mathbb{F}_{2^{mg}}^{*}$ not covered by the list of
	differences of $\mathcal{F}$ and let $\mathcal{G}$ be the set of cosets
	of $G$ in $\mathbb{F}_{2^{mg}}^{*}$. Then, by Theorem~\ref{RDF->GDD},
	the triple $\mathcal{D}=(\mathbb{F}_{q}^{n},\mathcal{G},dev
	\mathcal{F})$ is a cyclic $(2^{mg}-1,2^{g}-1,2^{k}-1,\lambda )$-GDD. Now
	note that $G$ is the multiplicative group of the subfield of order
	$2^{g}$ of $\mathbb{F}_{q}^{mg}$. Hence, adding 0 to each member of
	$\mathcal{G}$ we get the so-called \emph{regular} or
	\emph{Desarguesian} $g$-spread. Also, each block of $dev\mathcal{F}$ is
	of the form $xB$ with $x\in \mathbb{F}_{2^{n}}^{*}$ and
	$B\in {\mathcal{F}}$. On the other hand, by Definition~\ref{RDF_2}, we
	have that $B \ \cup \ \{0\}$ is a subspace of $\mathbb{F}_{2^{n}}$ so
	that $xB \ \cup \ \{0\}$ is a subspace of $\mathbb{F}_{2^{n}}$ as well.
	Thus every block of $\mathcal{D}$ is a subspace of $\mathbb{F}_{2}
	^{n}$ deprived of the zero vector. We conclude that $\mathcal{D}$ can
	be seen as a $(mg,g,k,\lambda )_{2}$ design by Remark~\ref{equivalenceGDD}.
\end{proof}

The above proposition will allow us to get a cyclic $(n,3,3,7)_{2}$-GDD
for every $n\equiv 3$ (mod 6).

%s2 #&#
\section{Revisiting and improving Thomas' result on 2-$(n,3,7)$ designs over $\mathbb{F}_{2}$}

Here we obtain a $(n,3,7)_{2}$ difference family for any positive odd
integer $n$. Thus, in view of Proposition~\ref{DF_2->design_2}, we prove
the following.
%
%t2.1 #&#
\begin{thm}
	\label{Thomas+}
	There exists a cyclic $2$-$(n,3,7)$ design over $\mathbb{F}_{2}$ for
	every odd positive integer~$n$.
\end{thm}

The above result was already obtained by Thomas \cite{T} in the
hypothesis that $\gcd (n,6)=1$. We first need to recall how the
solvability of a quadratic equation over $\mathbb{F}_{2^{n}}$ can be
established using the \emph{absolute trace} of $\mathbb{F}_{2^{n}}$.
This is the function $
\displaystyle
Tr: x\in \mathbb{F}_{2^{n}}\longrightarrow \sum _{i=0}^{n-1}x^{2^{i}}
\in \mathbb{F}_{2}$. Some elementary properties of this function which
could be useful later are the following:
\begin{itemize}%
	\item[] $Tr(x)+Tr(y)=Tr(x+y)$ for all $x, y \in \mathbb{F}_{2^{n}}$;
	\item[] $Tr(x^{2})=Tr(x)$ for all $x \in \mathbb{F}_{2^{n}}$;
	\item[] $Tr(1)=0$ or 1 according to whether $n$ is even or odd,
	respectively.
\end{itemize}

Here is the well known result concerning quadratic equations in a finite
field of characteristic two (see, e.g., \cite{H}).
%
%l2.2 #&#
\begin{lem}
	\label{equation}
	Let $ax^{2}+bx+c=0$ be a quadratic equation in $\mathbb{F}_{2^{n}}$ and
	let $m$ be the number of its distinct solutions in the same field. We
	have:
	\begin{itemize}%
		\item[] $m=1$ if and only if $b=0$;
		\item[] $m=2$ if and only if $b\neq 0$ and $Tr({ac\over b^{2}})=0$;
		\item[] $m=0$ if and only if $b\neq 0$ and $Tr({ac\over b^{2}})=1$.
	\end{itemize}
\end{lem}

The following fact is an immediate consequence of the above lemma.
%
%l2.3 #&#
\begin{lem}
	\label{trace}
	Let $ax^{2}+bx+c=0$ and $\alpha x^{2}+\beta x+\gamma =0$ be two
	quadratic equations in $\mathbb{F}_{2^{n}}$ with $b\beta \neq 0$.
	Exactly one of these equations is solvable in $\mathbb{F}_{2^{n}}$ if
	and only if
	$Tr({ac\over b^{2}})+Tr({\alpha \gamma \over \beta ^{2}})=1$.
\end{lem}

We are now ready to prove our main result.

%t2.4 #&#
\begin{thm}
	\label{main}
	There exists a $(n,3,7)_{2}$ difference family for every positive odd
	integer $n$.
\end{thm}
\begin{proof}
	We first associate with every $x\in \mathbb{F}_{2^{n}}^{*}\setminus
	\{1\}$ the subspace $S_{x}$ of $\mathbb{F}_{2}^{n}$ generated by $1$,
	$x$ and $x^{2}$. Note that these three elements are independent since,
	in the opposite case, we would have $x^{2}+x+1=0$ which implies
	$x^{3}=1$. This would mean that $x$ has order 3 in $\mathbb{F}_{2^{n}}
	^{*}$ so that $2^{n}-1$ would be divisible by 3 contradicting the
	hypothesis that $n$ is odd. Thus $S_{x}$ has dimension three. Now set
	$B_{x}:= S_{x}\setminus \{0\}$, hence
	%
	%e2.1 #&#
	\begin{equation}
	\label{Bx}
	B_{x}=\{1,x,x^{2},x+1,x^{2}+1,x^{2}+x,x^{2}+x+1\}.
	\end{equation}
	Note that $B_{x}=B_{x+1}$ for every $x$. It is convenient, anyway, to
	consider $B_{x}$ and $B_{x+1}$ as distinct blocks. Now consider the
	collection
	\begin{equation*}
		\mathcal{F}:=\{B_{x} \ | \ x\in \mathbb{F}_{2^{n}}^{*}\setminus \{1\}
		\}\end{equation*}
	and, for any $t\in \mathbb{F}_{2^{n}}^{*}\setminus \{1\}$, let
	$m(t)$ be the multiplicity of $t$ in $\Delta {\mathcal{F}}$.
	
	Let $\delta _{ij}(x)$ be the $(i,j)$ entry in the following table
	\medskip
	\begin{center}
		\renewcommand\arraystretch{1.4}
		\begin{tabular}{ccccccc}
			\hline $-$ & ${1\over x}$ & ${1\over x^{2}}$ & ${1\over x+1}$ & ${1\over x^{2}+1}$ & ${1\over x^{2}+x}$ & ${1\over x^{2}+x+1}$\\
			$x$ & $-$ & ${1\over x}$ & ${x\over x+1}$ & ${x\over x^{2}+1}$ & ${1\over x+1}$ & ${x\over x^{2}+x+1}$ \\
			$x^{2}$ & $x$ & $-$ & ${x^{2}\over x+1}$ & ${x^{2}\over x^{2}+1}$ & ${x\over x+1}$ & ${x^{2}\over x^{2}+x+1}$ \\
			$x+1$ & ${x+1\over x}$ & ${x+1\over x^{2}}$ & $-$ & ${1\over x+1}$ & ${1\over x}$ & ${x+1\over x^{2}+x+1}$ \\
			$x^{2}+1$ & ${x^{2}+1\over x}$ & ${x^{2}+1\over x^{2}}$ & ${x+1}$ & $-$ & ${x+1\over x}$ & ${x^{2}+1\over x^{2}+x+1}$ \\
			$x^{2}+x$ & $x+1$ & ${x+1\over x}$ & $x$ & ${x\over x+1}$ & $-$ & ${x^{2}+x\over x^{2}+x+1}$ \\
			$x^{2}+x+1$ & ${x^{2}+x+1\over x}$ & ${x^{2}+x+1\over x^{2}}$ & ${x^{2}+x+1\over x+1}$ & ${x^{2}+x+1\over x^{2}+1}$ & ${x^{2}+x+1\over x^{2}+x}$ & $-$ \\
			\hline
		\end{tabular}
	\end{center}
	
	\medskip
	\noindent
	representing the list $\Delta B_{x}$ of quotients of $B_{x}$. More
	precisely, $\delta _{ij}(x)$ is the quotient between the $i$-th and the
	$j$-th element of $B_{x}$ in the ordering of (\ref{Bx}). For every
	$t\in \mathbb{F}_{2^{n}}^{*}\setminus \{1\}$, let $m_{ij}(t)$ be the
	number of distinct solutions in $\mathbb{F}_{2^{n}}$ of the equation
	\begin{equation*}
		E_{ij}(t): \ \delta _{ij}(x)=t
	\end{equation*}
	in the unknown $x$. It is clear that we have
	\begin{equation*}
		m(t)=\sum _{i\neq j}m_{ij}(t).
	\end{equation*}
	Note that $E_{ij}(t)$ can be rewritten as a quadratic equation
	$ax^{2}+bx+c=0$ with $b\neq 0$ for any pair $(i,j)$ belonging to the
	18-set
	\begin{multline*}
		I=\{(1,6),(1,7),(2,5),(2,7),(3,4),(3,7),(4,3),(4,7),(5,2),\\
		(5,7),(6,1),(6,7),(7,1),(7,2),(7,3),(7,4),(7,5),(7,6)\}.
	\end{multline*}

	Thus $m_{ij}(t)=0$ or 2 for every $(i,j)\in I$. On the other hand, it
	is easily seen that for all twenty-four pairs $(i,j)\notin I$ we have
	$m_{ij}(t)=1$ since in this case $E_{ij}(t)$ becomes either an equation
	of the first degree or an equation of the form $ax^{2}+c=0$. It follows
	that $m(t)=24+2\cdot r(t)$ where $r(t)$ is the number of equations
	$E_{ij}(t)$ with $(i,j)\in I$ which are solvable in
	$\mathbb{F}_{2^{n}}$. We want to prove that $r(t)=9$ for every $t$. For
	this, we have to show that it is possible to match the eighteen
	equations $E_{ij}(t)$ with $(i,j)\in I$ in such a way that, in each
	match, only one equation is solvable in $\mathbb{F}_{2^{n}}$. Using
	Lemma~\ref{trace} and taking into account the mentioned properties of
	the trace function, the reader can easily check that such a good
	matching is the following.
	
	%\small
	\begin{center}
		\renewcommand\arraystretch{1.4}
		\begin{tabular}{ll}
			\hline $E_{61}(t): \ x^{2}+x+t=0$ & $E_{71}(t): \ x^{2}+x+t+1=0$ \\
			$E_{16}(t): \ tx^{2}+tx+1=0$ & $E_{17}(t): \ tx^{2}+tx+t+1=0$ \\
			$E_{52}(t): \ x^{2}+tx+1=0$ & $E_{37}(t): \ (t+1)x^{2}+tx+t=0$ \\
			$E_{72}(t): \ x^{2}+(t+1)x+1=0$ & $E_{27}(t): \ tx^{2}+(t+1)x+t=0$ \\
			$E_{43}(t): \ tx^{2}+x+1=0$ & $E_{73}(t): \ (t+1)x^{2}+x+1=0$ \\
			$E_{74}(t): \ x^{2}+(t+1)x+t+1=0$ & $E_{47}(t): \ tx^{2}+(t+1)x+t+1=0$ \\
			$E_{75}(t): \ (t+1)x^{2}+x+t+1=0$ & $E_{25}(t): \ tx^{2}+x+t=0$ \\
			$E_{76}(t):\ (t+1)x^{2}+(t+1)x+1=0$ & $E_{67}(t): \ (t+1)x^{2}+(t+1)x+t=0$ \\
			$E_{34}(t): \ x^{2}+tx+t=0$ & $E_{57}(t): \ (t+1)x^{2}+tx+t+1=0$ \\
			\hline
		\end{tabular}
	\end{center}

	\medskip
	Consider, as an example, the third of the above pairs $(E_{52}(t),E
	_{37}(t))$. By Lemma~\ref{equation}, $E_{52}(t)$ is solvable if and only
	if $Tr({1\over t^{2}})=0$, while $E_{37}(t)$ is solvable if and only if
	$Tr({t+1\over t})=0$. Now, by the properties of the trace function, we
	have:
	\begin{equation*}
		Tr\biggl ({1\over t^{2}}\biggl )+Tr\biggl ({t+1\over t}\biggl )=Tr\biggl (
		{1\over t}\biggl )+Tr\biggl ({t+1\over t}\biggl )=Tr(1)=1.
	\end{equation*}
	Hence, by Lemma~\ref{trace}, only one of the two equations $E_{52}(t)$
	and $E_{37}(t)$ is solvable in $\mathbb{F}_{2^{n}}$.
	
	We conclude that $m(t)=24+2\cdot 9=42$ for any $t\in \mathbb{F}_{2^{n}}
	^{*}\setminus \{1\}$. This means that $\mathcal{F}$ is a $(n,3,42)_{2}$
	difference family.
	
	Now consider the 2-regular graph $\Gamma $ with vertex-set $
	\mathbb{F}_{2^{n}}^{*}\setminus \{1\}$ where the two neighbors of any
	vertex $x$ are $x+1$ and ${1\over x}$. It is clear that the connected
	components of $\Gamma $ are all the hexagons of the form
	
	
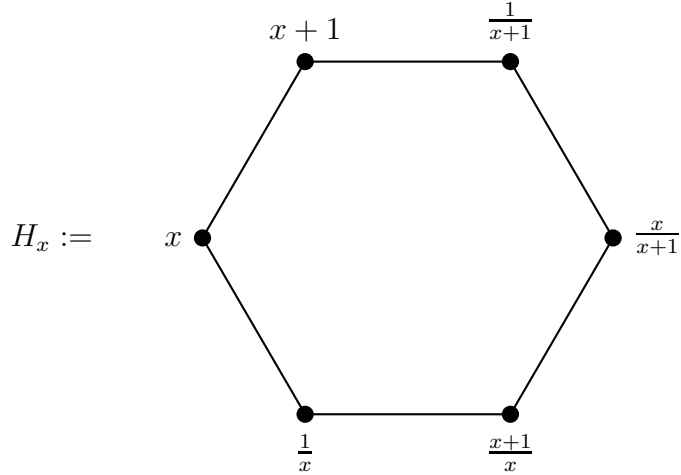
\begin{center}
		\begin{tikzpicture}[thick]
		\newdimen\R
		\R =2.7cm
		\foreach \x/\l in
		{ 60/\,, 120/\,, 180/\,, 240/\,, 300/\,, 360/\, }
		\draw[postaction={decorate}] ({\x-60}:\R) -- node[auto,swap]{\l} (\x:\R);
		\foreach \x/\l/\p in
		{ 60/{$\frac{1}{x+1}$}/above,
			120/{$x+1$}/above,
			180/{$H_x:=\quad\quad x$}/left,
			240/{$\frac{1}{x}$}/below,
			300/{$\frac{x+1}{x}$}/below,
			360/{$\frac{x}{x+1}$}/right
		}
		\node[inner sep=2pt,circle,draw,fill,label={\p:\l}] at (\x:\R) {};
		\end{tikzpicture}
	\end{center}
	
	We note that all blocks $B_{y}$ with $y$ lying in the hexagon
	$H_{x}$ are in the same $\mathbb{F}_{2^{n}}^{*}$-orbit. Indeed we
	already commented that $B_{x}$ and $B_{x+1}$ coincide. Also, the reader
	can easily check that $B_{1/x}={1\over x^{2}}\cdot B_{x}$. It follows
	that all six blocks associated with the vertices of any hexagon of
	$\Gamma $ produce the same list of quotients. Then, considering that
	$\mathcal{F}$ is a $(n,3,42)_{2}$ difference family, it is evident that
	if $X$ is a complete system of representatives for the hexagons of
	$\Gamma $, then $\mathcal{F}':=\{B_{x} \ | \ x\in X\}$ is a
	$(n,3,7)_{2}$ difference family. The assertion follows.
\end{proof}

In the following we will keep the same notation that we used in the
above proof. It is clear that the design constructed in the above
theorem does not depend on the system $X$ of representatives for the
hexagons of $\Gamma $. Recall in fact that $B_{x}=B_{x+1}$ and that
$B_{x}=x^{2}\cdot B_{1/x}$ so that the blocks associated with the
vertices of $H_{x}$ have all the same development.

When $n\equiv \pm 1$ (mod 6), that is the case also considered by
Thomas, our design coincides with his design. Indeed our blocks are
exactly what he calls \emph{special triangles}. The two descriptions are
different since while Thomas' approach is essentially geometric, our
approach is purely algebraic.

Now, given $x\in \mathbb{F}_{2^{n}}^{*}\setminus \{1\}$, we want to show
that a block $B_{y}$ of the $(n,3,7)_{2}$ difference family
$\mathcal{F}$ is in the same $\mathbb{F}_{2^{n}}^{*}$-orbit of
$B_{x}$ if and only if $y$ is in $V(H_{x})$, the set of vertices of the
hexagon $H_{x}$. The ``if-part'' has been already shown in the proof of
Theorem~\ref{main}. Let us prove the ``only-if-part''. Assume that
$B_{y}$ is in the same $\mathbb{F}_{2^{n}}^{*}$-orbit of $B_{x}$ so that
there exists a non-zero field element $t$ such that $B_{y}=tB_{x}$. Such
equality implies that $
\begin{cases}
1=tf_{0}
\cr
y=tf_{1}
\cr
y^{2}=tf_{2}
\end{cases}
$ with $(f_{0},f_{1},f_{2})$ a triple of distinct elements of
$B_{x}$. In its turn the above system implies that $f_{0}f_{2}+f_{1}
^{2}=0$. Considering the form of the elements of $B_{x}$, we see that
\begin{equation*}
	f_{0}f_{2}+f_{1}^{2}=c_{0}+c_{1}x+c_{2}x^{2}+c_{3}x^{3}+c_{4}x^{4}
\end{equation*}
for a suitable quintuple $(c_{0},\dots ,c_{4})$ of elements of
$\mathbb{F}_{2}$, namely $x$ is a zero of the polynomial $p(z)=\sum
_{i=0}^{4} c_{i}z^{i}\in \mathbb{F}_{2}[z]$. First note that
$p(z)$ is the null polynomial -- namely we have $c_{i}=0$ for each
$i$ -- exactly when $(f_{0},f_{1},f_{2})$ and $y$ are as follows:
\begin{center}
	\renewcommand\arraystretch{1.4}
	\begin{tabular}{ccccccc}
		\hline $f_{0}$ & $1$ & $x^{2}$ & $1$ & $x^{2}+1$ & $x^{2}$ & $x^{2}+1$\\
		$f_{1}$ & $x$ & $x$ & $x+1$ & ${x+1}$ & $x^{2}+x$ & $x^{2}+x$ \\
		$f_{2}$ & $x^{2}$ & $1$ & $x^{2}+1$ & $1$ & $x^{2}+1$ & $x^{2}$ \\
		$y$ & $x$ & ${1\over x}$ & $x+1$ & ${1\over x+1}$ & ${x+1\over x}$ & ${x\over x+1}$ \\
		\hline
	\end{tabular}
\end{center}

So we see that in this case $y$ is a vertex of $H_{x}$.

Now assume that $p(z)$ has degree $d$ with $1\leq d\leq 4$. In this case
the zeros of $p(z)$ lying in $\mathbb{F}_{2^{n}}$ are in the subfield
of order $2^{\gcd (n,d)}$. Considering that $n$ is odd we have either
$\gcd (n,d)=1$ or $\gcd (n,d)=3$. In the first case $x$ should lie in
the subfield of order 2, i.e., $x\in \{0,1\}$ which is absurd. In the
second case $x$ would be in the subfield $\mathbb{K}$ of order 8 and
consequently both $B_{x}$ and $V(H_{x})$ coincide with $\mathbb{K}
^{*}\setminus \{1\}$. It immediately follows that $y$ is also in
$\mathbb{K}$ and then $y\in V(H_{x})$.

It is clear that the stabilizer of any $B_{x}$ is a common divisor of
$2^{n}-1$ and $|B_{x}|=7$. Thus it is always trivial when $n\equiv
\pm 1$ (mod 6). Instead, for $n\equiv 3$ (mod 6), $B_{x}$ has
non-trivial stabilizer if and only if $B_{x}$ is the multiplicative
group of the subfield $\mathbb{K}$ of order 8.

The above considerations, together with Remark~\ref{simple}, allow us
to state the following.
%
%r2.1 #&#
\begin{rem}
	The cyclic $(n,3,7)_{2}$ design constructed in Theorem~\ref{main} is
	simple if and only if $n\equiv \pm 1$ (mod 6).
\end{rem}

When $n\equiv 3$ (mod 6), that is the case not considered by Thomas,
$\mathbb{F}_{2^{n}}$ has a subfield $\mathbb{K}$ of order $8$ and we
already commented that for every $x\in \mathbb{K}^{*}$ the block
$B_{x}$ coincides with $\mathbb{K}^{*}$ (which is also the vertex-set
of $H_{x}$). Thus, if $y$ is the representative of $X$ in $\mathbb{K}
^{*}$, then $\mathcal{F}'$ is a $(2^{n}-1,7,7)$ difference family in
$\mathbb{F}_{2^{n}}^{*}$ with a base block $B_{y}$ that is a subgroup
of $\mathbb{F}_{2^{n}}^{*}$. It follows, by Proposition~\ref{mk,k,k},
that $\mathcal{F}'':=\mathcal{F}'\setminus \{B_{y}\}$ is a
$(n,3,3,7)_{2}$ difference family and then, by Proposition~\ref{RDF_2->GDD_2}, we can state the following.

%t2.5 #&#
\begin{thm}
	There exists a cyclic and simple $(n,3,3,7)_{2}$ group divisible design
	for every integer $n\equiv 3$ \textup{(}mod $6$\textup{)}.
\end{thm}

As far as we know this the first infinite family of cyclic GDDs over a
finite field.

\section*{Acknowledgements}
This work has been performed under the auspices
of the G.N.S.A.G.A. of the C.N.R. (National Research
Council) of Italy.
The first author carried out this research within the project  
``Disegni combinatori con un alto grado di simmetria", supported by Fondo Ricerca di Base, 2015, of Universit\`a degli Studi di Perugia. The second author is supported in part by the Croatian Science Foundation under the project 6732.

\end{document}